\documentclass[journal]{IEEEtran}

\usepackage{graphicx}      
\usepackage{amsmath}
\newtheorem{remark}{Remark}
\newtheorem{theorem}{Theorem}
\newtheorem{lemma}{Lemma}
\usepackage{amssymb}
\usepackage{color}
\usepackage{mathtools}

\usepackage{float}
\graphicspath{ {pics/} } 
\newenvironment{proof}{{Proof:}}{\hfill$\square$}

\allowbreak 
\IEEEoverridecommandlockouts
\begin{document}

\title{Bilateral Boundary Control of Moving Shockwave in LWR Model of Congested Traffic} 

\author{Huan Yu,~\IEEEmembership{Student Member,~IEEE,} Mamadou Diagne,~\IEEEmembership{Member,~IEEE,} Liguo Zhang~\IEEEmembership{Member,~IEEE,} \\Miroslav Krstic,~\IEEEmembership{Fellow,~IEEE}
\thanks{Huan Yu and Miroslav Krstic are with the Department of Mechanical and Aerospace Engineering,
	University of California, San Diego, 9500 Gilman Dr, La Jolla, CA 92093.
	{Email: huy015@ucsd.edu, krstic@ucsd.edu}}%
\thanks{Mamadou Diagne is with the Department of Mechanical Aerospace and Nuclear Engineering, Rensselaer Polytechnic Institute, New York, 12180, USA. {Email: diagnm@rpi.edu}}
\thanks{Liguo Zhang is with the School of Electronic Information and Control Engineering, Beijing University of Technology, Beijing, 100124, China. Email: zhangliguo@bjut.edu.cn
}}

\maketitle	
	
\begin{abstract}                
We develop backstepping state feedback control to stabilize a moving shockwave in a freeway segment under bilateral boundary actuations of traffic flow. A moving shockwave, consisting of light traffic upstream of the shockwave and heavy traffic downstream, is usually caused by changes of local road situations. The density discontinuity travels upstream and drivers caught in the shockwave experience transitions from free to congested traffic. Boundary control design in this paper brings the moving shockwave front to a static setpoint position, hindering the upstream propagation of traffic congestion. The traffic dynamics are described with Lighthill-Whitham-Richard (LWR) model, leading to a system of two first-order hyperbolic partial differential equations (PDEs). Each represents the traffic density of a spatial domain segregated by the moving interface. By Rankine-Hugoniot condition, the interface position is driven by flux discontinuity and thus governed by a PDE state dependent ordinary differential equation (ODE). For the PDE-ODE coupled system. the control objective is to stabilize both the PDE states of traffic density and the ODE state of moving shock position to setpoint values. Using delay representation and backstepping method, we design predictor feedback controllers to cooperatively compensate state-dependent input delays to the ODE. From Lyapunov stability analysis, we show local stability of the closed-loop system in $H^1$ norm. The performance of controllers is demonstrated by numerical simulation.	
\end{abstract}



\section{Introduction}\label{intro}	

Consider a common phenomenon in freeway traffic when there is a moving shockwave consisting of light traffic upstream of the shockwave and heavy traffic downstream. The shockwave conserves traffic flow at the interface of discontinuity and is caused by local changes of road situations like uphill and downhill gradients, curves, change of speed limits. The upstream propagation of the moving shockwave causes more and more vehicles entering into the congested traffic. The abrupt transition from free to congested traffic at the moving interface leads to unsafe driving conditions and increased fuel consumptions. It is of great importance if we can halt the upstream propagation and drive the moving interface to a desirable location where the traffic congestion could be discharged by traffic management infrastructures on freeways. Ramp metering and varying speed limit are most widely used to control traffic flux or velocity from the boundary of a stretch of freeway so that desirable traffic states could be achieved for the inner domain of the freeway segment.

In developing boundary control strategies through ramp metering and varying speed limit, many recent efforts ~\cite{Belletti:15},\cite{Karafyllis:18},\cite{Huan:19},\cite{Huan1:18},\cite{Huan:full},\cite{Huan:adapt},\cite{Huan:18},\cite{Liguo:19} are focused on macroscopic traffic models governed by PDE system. These model-based controllers  regulate the evolution of traffic densities and velocities in order to dissipate traffic congestions on freeways. For instance, \cite{Huan:19},\cite{Huan1:18} achieve $L^2$ norm stabilization of stop-and-go traffic by nonlinear second-order PDE traffic model using boundary control. 

Traffic discontinuity can be caused by various inhomogeneities of freeway or vehicles. Some studies consider it as a moving traffic flux constraint~\cite{Monache:14},\cite{Villa:16} due to a reduction of road capacity. Slow moving vehicles, also known as moving bottlenecks, are represented in~\cite{Borsche:12},\cite{Lattanzio:11},\cite{Huan:18} with ODEs governing the velocity of slow vehicles. These are out of the scope of this paper and relevant to the controllability problem with boundary actuation. In this paper, we consider the situation where road capacity is conserved but shockwaves form due to uphills, downhills, and curves of the road. Higher density traffic appears downstream of the shockwave front and the front of density discontinuity keeps moving upstream, driven by the flux discontinuity. The upstream propagation of the moving shockwave causes traffic congestion forming up on a freeway.

In this work, we adopt the seminal Lighthill, Whitham and Richards (LWR) model to describe the traffic dynamics of the moving shockwave problem. The LWR model is a first-order, hyperbolic macroscopic PDE model of traffic density. It is simple yet very powerful to describe the formation, dissipation and propagation of traffic shockwaves on a freeway. The moving shockwave consists of upstream, downstream traffic and a moving interface. The upstream and downstream traffic densities are governed by LWR PDE models and the interface position is governed by Rankine-Hugoniot jump condition, leading to a density state-dependent nonlinear ODE. Therefore, we are dealing with a PDE-ODE coupled system, where ODE state is dependent on PDE states at the moving interface. The traffic flow is actuated at both boundaries of a freeway segment and can be realized with ramp-metering. The control objective is to drive the moving interface to certain location and traffic states to steady values through bilateral boundary controls.

Boundary control of PDE with state-dependent ODE systems has been intensively studied over the past few years. Backstepping control design method is used in solving these problems. In parabolic PDE system, the problem is known as Stefan problem with application to control of screw extruder for 3D Printing~\cite{Koga:18} and arctic sea ice temperature estimation~\cite{Koga:17}. In hyperbolic PDE system, theoretical results have been studied by \cite{Nik:13a},\cite{Nik:13b},\cite{Nik:17},\cite{Diagne:17},\cite{Tsuba:16}. With application, \cite{Buisson:18} develops boundary control piston position in inviscid gas and \cite{Diagne:15} develops the control of a mass balance in screw extrusion process. Other applications include vibration suppression of mining cable elevator~\cite{Wang:18}, control of Saint-Venant equation with hydraulic jumps~\cite{Coron:07}. However, the application of the methodology in traffic problem has never been discussed before.

The contribution of this paper is twofold. This is the very first theoretical result on control of two PDE state-dependent input delays to ODE. Predictor-based state feedback design approach is adopted following~\cite{Diagne:17},\cite{Tsuba:16}. In fact, \cite{Tsuba:16} shows a predictor feedback design for multiple constant delayed inputs to linear time-invariant systems while   \cite{Diagne:17} considers a single implicitly defined state-dependent input delay to nonlinear time-invariant systems alternatively written as a PDE-ODE cascade system. In this work, we firstly present the predictor feedback design for two PDE states dependent input delays to ODE. 
On the other hand, control problem of traffic moving shockwave has never been addressed before to author's best knowledge. 

The outline of this paper: 	
we introduce the LWR model to describe the moving shockwave problem. Then we linearized the coupled PDE-ODE model around steady states. The predictor state feedback control design follows and using Lyapunov analysis, we prove the local exponential stability of the closed-loop system. Model validity is guaranteed with the control design. In the end, the result is validated with numerical simulations.

\section{Problem statement}
\begin{figure}[t!]
	\centering
	\includegraphics[width=6.8cm]{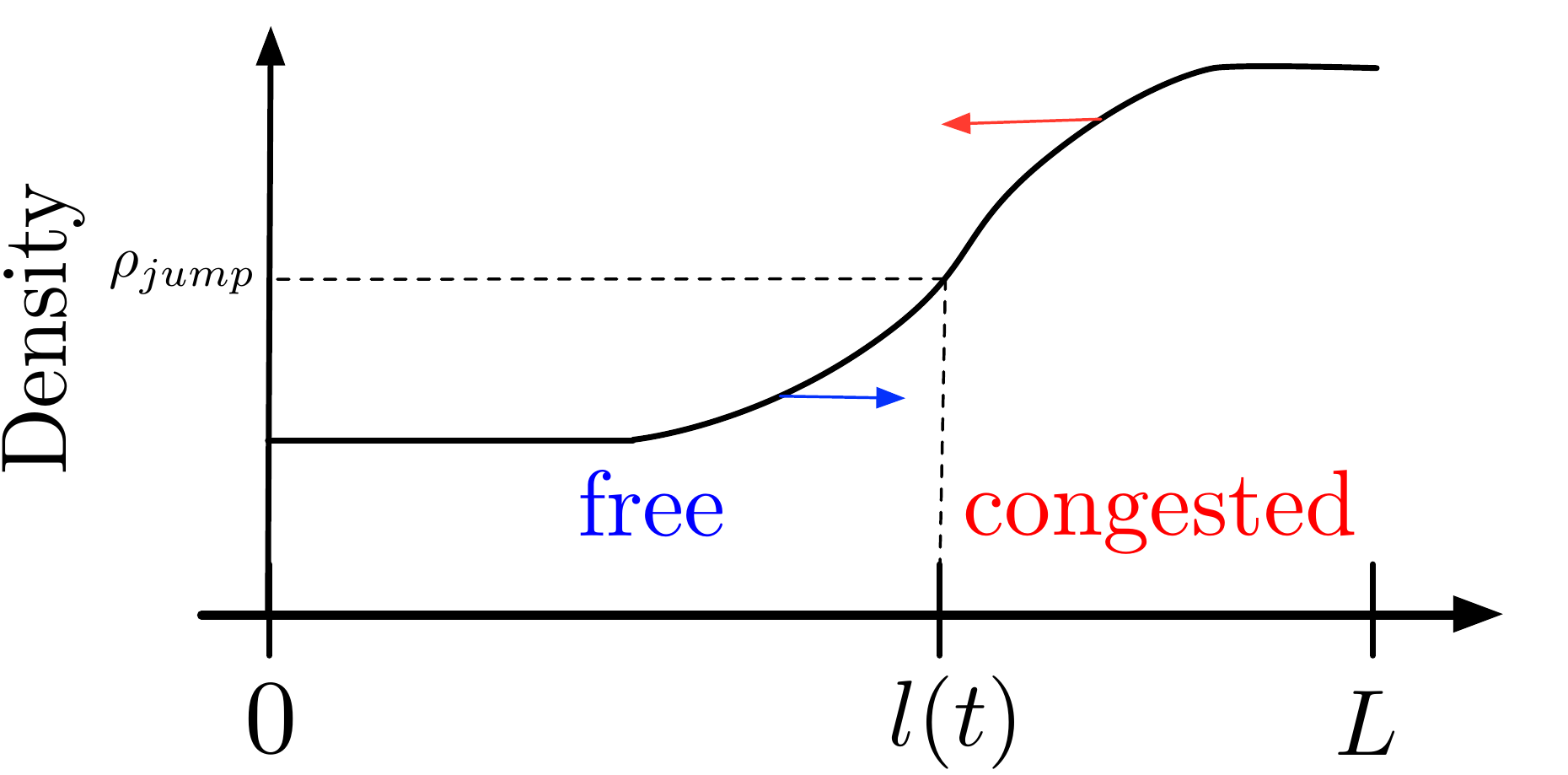}
	\caption{Traffic moving shockwave front on freeway, the arrows represent propagation directions of density variations. In LWR model, the propagation directions are given by the characteristic speeds of density $Q'(\rho)$. }\label{fig1}
\end{figure} 

The moving shockwave front is the head of a shockwave, segregating traffic on a segment of freeway into two different schemes. The upstream traffic of the shockwave front is in free regime and the downstream is in congested regime, as shown in Fig.1. The traffic densities are described with the first-order macroscopic LWR model. 

\subsection{LWR traffic model}
In LWR model, traffic density $\rho(x,t)$ is governed by the following first-order nonlinear hyperbolic PDE, where $x\in[0,L]$, $t\in[0,\infty)$, 
\begin{align}
\partial_t \rho + Q'(\rho)\partial_x \rho =& 0,\label{lwr}
\end{align}
where $Q(\rho)$ is a fundamental diagram which shows the relation of equilibrium density and traffic flux. The fundamental diagram $Q(\rho)$ is defined as $Q(\rho) = \rho V(\rho)$. The equilibrium velocity $V(\rho)$ is a decreasing function of density. We choose the following Greenshield's model for $V(\rho)$ in which velocity is a linear decreasing function of density.
\begin{align}
	V(\rho) = v_m \left(1 - \frac{\rho}{\rho_m}\right).\label{Green}
\end{align}
where ${v_{m}}$ is the maximum speed, ${\rho_{m}}$ is the maximum density. Greenshield's model $V(\rho)$ yields that the fundamental diagram $Q(\rho)$ is a quadratic map, shown in figure Fig. \ref{fig2}. The jump density $\rho_{\rm jump}$ segregates densities into two sections, the density smaller than $\rho_{jump}$ is defined as free-regime while the density greater than $\rho_{\rm jump}$ is defined as congested regime. 
\begin{figure}
\centering
 \includegraphics[width=6cm]{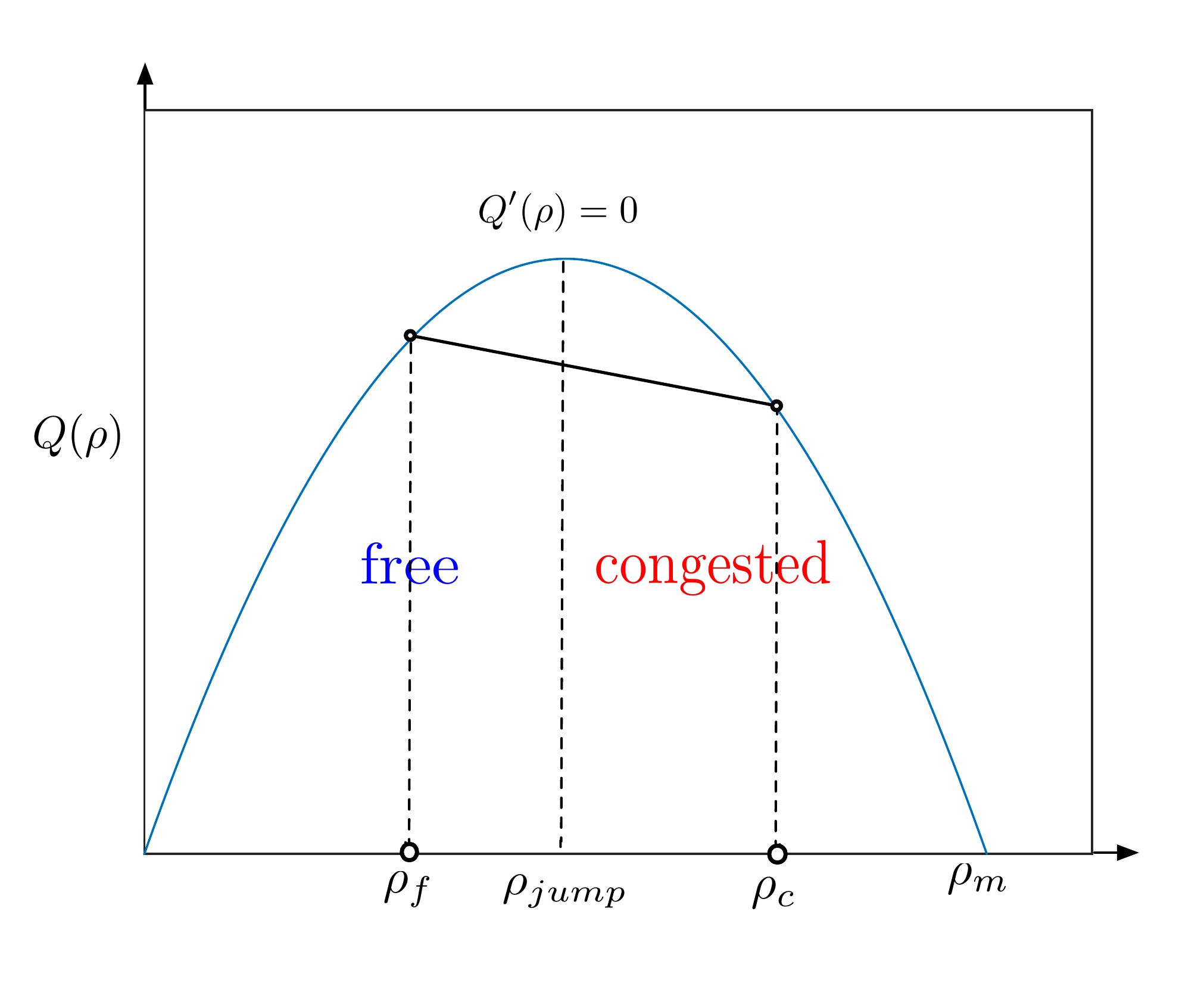}
 \caption{Fundamental digram of traffic density and traffic flux relation}\label{fig2}
\end{figure}

In the LWR PDE \eqref{lwr}, density variations propagate with the characteristic speed $ Q'(\rho)$. 
The free regime with light traffic, equivalently, $\rho_{\rm f} < \rho_{\rm jump}$, has its density variations transported  downstream with
\begin{align}
	 Q'(\rho)|_{\rho=\rho_{\rm f}} = V(\rho_{\rm f}) + \rho_{\rm f} V'(\rho_{\rm f}) > 0,
\end{align}
while the congested regime with denser traffic, namely, $\rho_{\rm c} > \rho_{\rm jump}$ has its density variations transported upstream with
\begin{align}
	 Q'(\rho)|_{\rho=\rho_{\rm c}} = V(\rho_{\rm c}) + \rho_c V'(\rho_{\rm c}) < 0.
\end{align}
As shown in figure Fig. \ref{fig1}, the moving shockwave considered here is the shock of a traffic wave which physically represents the discontinuity of density. The congested traffic density propagates upstream while the light traffic density propagates downstream. Therefore, the upstream front of the shockwave becomes steeper in propagation and  eventually, the gradient $\partial_x \rho$ tends to be infinity  \cite{Treiber:14}. In this context, drivers located  in the upstream front of the shock will experience transition from free to congested traffic. The position of the shockwave front is later defined by an ODE according to Rankine-Hugoniot condition.

\subsection{ Moving shockwave model}
The moving shockwave model consists of upstream, downstream traffic densities and a moving interface located at the  density discontinuity spatial coordinate.
The dynamics of the upstream free traffic, the downstream congested traffic and the position of the moving interface are presented below, respectively. 

Define the traffic density of the congested regime as $\rho_{\rm c}(x,t)$ for $x\in [0,l(t)]$, $t\in [0, +\infty]$, and the free regime as $\rho_{\rm f}(x,t)$, for $x\in [l(t), L]$, $t\in [0, +\infty] $, the LWR model that describes the traffic is given by
\begin{align}
\partial_t \rho_{\rm f}+ \partial_x (\rho_{\rm f} v_{\rm f})=& 0, \quad x\in [0,l(t)] \label{f}  \\
\partial_t \rho_{\rm c}+ \partial_x (\rho_{\rm c} v_{\rm c})=& 0, \quad x\in [l(t), L]  \label{c} 
\end{align}
where $l(t) \in [0,L]$ is the location of moving interface. The density and velocity relation is given by Greenshield's model in \eqref{Green}, $(i = \rm f, c)$,
\begin{align} 
v_i(x,t)=& V_i(\rho_i(x,t))= {v_{m}} \left(1-\frac{\rho_i(x,t)}{\rho_{m}}\right). \label{vr}
\end{align}
Due to the flux discontinuity at the moving boundary, a traveling vehicle leaves the free regime to enter the congested regime. Dynamics of moving interface $l(t)$ is derived under the Rankine-Hugoniot condition which guarantees that the mass of traffic flow is conserved at the moving interface. The upstream propagation of the shockwave front is driven by the flux discontinuity.  

\begin{align}\label{Interface}
\dot l(t)=& \frac{\rho_{\rm c}(l(t),t)v_{\rm c}(l(t),t)-  \rho_{\rm f}(l(t),t)v_{\rm f}(l(t),t)}{\rho_{\rm c}(l(t),t)-\rho_{\rm f}(l(t),t)},
\end{align}
where the initial position of the shockwave front $0<l(0)<L$. {The following inequalities for initial conditions of PDEs \eqref{f},\eqref{c} are assumed }
\begin{align}
\rho_{\rm c}(l(0),0)v_{\rm c}(l(0),0) <& \rho_{\rm f}(l(0),0)v_{\rm f}(l(0),0),\\
\rho_{\rm c}(l(0),0)>&\rho_{\rm f}(l(0),0).
\end{align}
Initially, the  traffic  downstream the interface is denser but with a smaller flux which lets less vehicles to pass through while the traffic upstream is light and let more vehicles to come in the segment. With the above assumptions to hold, we obtain from \eqref{Interface} that $\dot l(0) < 0$.
The moving interface is traveling upstream and is driven by a flux difference induced by the density discontinuity.

Substituting density-velocity relation in \eqref{vr} into  \eqref{f},\eqref{c}, and  \eqref{Interface}, we have
two nonlinear PDEs and an ODE coupled system describing the dynamics of  $\rho_{\rm f}(x,t)$, $\rho_{\rm c}(x,t)$ and $l(t)$ given by
\begin{align}
\partial_t \rho_{\rm f}(x,t)=& - v_m\partial_x \left(\rho_{\rm f}(x,t)-\frac{\rho_{\rm f}^2(x,t)}{\rho_{m}}\right), \label{rf}\\
\partial_t \rho_{\rm c}(x,t)=&- v_m\partial_x \left(\rho_{\rm c}(x,t)-\frac{\rho_{\rm c}^2(x,t)}{\rho_{m}}\right), \label{rc}\\
 \dot l(t)
=&v_m - \frac{v_m}{\rho_m}(\rho_{\rm c}(l(t),t)+\rho_{\rm f}(l(t),t)). \label{lt}
\end{align}

%

\begin{remark}
	For model validity, we assume that there exists a constant $L>0$ such that the ODE state $l(t)$ satisfies
	\begin{align} 
	0&< l(t) < L, \label{boundl}
	\end{align}
so that \eqref{rf},\eqref{rc}, and \eqref{lt} are well-defined for $x\in [0, L]$, $t\in [0, +\infty] $.
We emphasize that the  proposed control law    needs to guarantee the above condition.
\end{remark}

Our control objective is to stabilize both free and congested regime traffic $\rho_i(x,t)$ to uniform steady states $\rho_i^\star$ and at the same time, the moving interface $l(t)$ to a desirable static setpoint $l^\star$. Therefore, the shockwave becomes standstill within the freeway segment instead of moving upstream.

We consider the following controlled boundary condition for the nonlinear coupled PDE-ODE system consisting of \eqref{rf}, \eqref{rc}, and \eqref{lt}
\begin{align}
\rho_{\rm f} (0,t)  & = U_{\rm in}(t) + \rho_{\rm f}^\star,\label{b1}\\
\rho_{\rm c}(L,t) & = U_{\rm out}(t) + \rho_{\rm c}^\star \label{b2},
\end{align}
where we control the incoming and outgoing density variations of the freeway segment $U_{\rm in}(t)$ and $U_{\rm out}(t)$. As mentioned in Section \ref{intro}, the control of density can be realized with on-ramp metering actuating the  flux at both boundaries:
\begin{align}
  q_{\rm in}(t) =& Q(\rho_{\rm f} (0,t)),\\
  q_{\rm out}(t) =& Q(\rho_{\rm c} (L,t)).
\end{align}

\section{ Linearized Model}\label{linear}
Now, we linearize the coupled PDE-ODE model $(\rho_{\rm f}(x,t), \rho_{\rm c}(x,t), l(t))$-system  defined in \eqref{rf},\eqref{rc} and \eqref{lt} around steady states and setpoint $(\rho_{\rm f}^\star, \rho_{\rm c}^\star, l^\star)$. The constant equilibrium  setpoint values are chosen so that the following conditions that ensure the  model validity hold
\begin{align}
0&<\rho_{\rm f}^\star < \rho_{\rm jump} <\rho_{\rm c}^\star< \rho_{m},\label{md1}\\ 
0&< l^\star <L.\label{md2}
\end{align}            
At steady-state, the flux equilibrium needs to be achieved for both sides of the moving interface. Hence,
\begin{align}\label{c1}
\rho_{\rm f}^\star V(\rho_{\rm f}^\star) = \rho_{\rm c}^\star V(\rho_{\rm c}^\star). 
\end{align}
Using condition \eqref{c1},  the quadratic fundamental diagram yields that
\begin{align}
\rho_{\rm f}^\star + \rho_{\rm c}^\star = \rho_m.\label{md3}
\end{align}
Define the  state deviations from the system reference  as
\begin{align}
\tilde \rho_i(x,t) = & \rho_i(x,t) - \rho_i^\star, \\
X(t) =& l(t) -l^\star,
\end{align}
where  $\dot X(t) = \dot l(t)$ is satisfied. Thus, the linearized PDE-ODE model \eqref{rf}-\eqref{lt}  with the  boundary conditions \eqref{b1}  and \eqref{b2} around the system reference $(\rho_{\rm f}^\star, \rho_{\rm c}^\star, l^\star)$  is defined as the  following  $(\tilde \rho_{\rm f}(x,t), \tilde \rho_{\rm c}(x,t), X(t))$-system 
\begin{align}
\partial_t \tilde \rho_{\rm f}(x,t) =&  -u \partial_x \tilde \rho_{\rm f}(x,t), \quad x \in [0,l(t)] \label{lf}\\
\partial_t \tilde \rho_{\rm c}(x,t)  =& u \partial_x \tilde \rho_{\rm c}(x,t),  \quad \;\;\;\;\; x \in [l(t),L]  \label{lc}\\
\tilde \rho_{\rm f}(0,t) =& U_{\rm in}(t) , \label{Uin}\\
\tilde \rho_{\rm c}(L,t) =& U_{\rm out}(t) , \label{Uout}\\
\dot X(t)= & -b\left(\tilde \rho_{\rm f}(l(t),t) + \tilde \rho_{\rm c}(l(t),t)\right), \label{mi}
\end{align}
where the transport speed is defined as 
\begin{align}
	u =v_{m} \left(1 - \frac{2 \rho_{\rm f}^\star}{\rho_{m}} \right),
\end{align} and satisfy $0< u < v_m.$
The constant coefficient $b$ in ODE is defined as $b = \frac{v_m}{\rho_m} >0$. The linearized model \eqref{lf}-\eqref{mi} is a PDE-ODE coupled system with bilateral boundary control inputs from inlet and outlet.

\section{Predictor-based control design}
In this section, we first introduce the  equivalent delay system  representation to the system  \eqref{lf}-\eqref{mi}. Then,  a  backstepping transformation is applied to obtain predictor-based state feedback controls to compensate the PDE    state-dependent delays to the ODE.  
 
\subsection{From coupled  PDE-ODE   to delay system  representation}
The system   \eqref{lf}-\eqref{mi} can be represented by an unstable ODE with two distinct state-dependent input delays. Introduce the following state-dependent delays for the two transport PDEs
\begin{align}
D_{\rm f}(t) &= \frac{ l(t)}{u},\\
D_{\rm c}(t) &= \frac{L - l(t)}{u},
\end{align}
where $l(t) = X(t) + l^\star$.
The PDE states are represented by
\begin{align}
\tilde \rho_{\rm f} (l(t),t) =& U_{\rm in}\left(t - D_{\rm f}(t)\right),\label{inputdelay1}\\
\tilde \rho_{\rm c}(l(t),t) =& U_{\rm out}\left(t - D_{\rm c}(t)\right)\label{inputdelay2},
\end{align}
where $U_{\rm in}(t)$ and $U_{\rm out}(t)$ are the boundary control inputs  defined in \eqref{Uin} and \eqref{Uout}. Substituting \eqref{inputdelay1} and \eqref{inputdelay2}   into the ODE \eqref{mi}, the following state-dependent   input delay system representation is derived
\begin{align}
\dot X(t)= & -b\left(U_{\rm in}(t - D_{\rm f}(X(t))) + U_{\rm out}(t - D_{\rm c}(X(t))\right).
\end{align}
\begin{remark}If the position of the moving shock front is close to the inlet half segment such that $l(t) \in \left[0,\frac{L}{2}\right],$ it holds that $\forall t\in [0,\infty), D_{\rm f}(t) \leq D_{\rm c}(t)$. As a result, delayed inlet control input $U_{\rm in}\left(t - D_{\rm f}(t)\right)$ reaches the moving shock front faster than delayed outlet control input $U_{\rm out}\left(t - D_{\rm c}(t)\right)$. If $l(t) \in \left[\frac{L}{2}, L\right],$ $\forall t\in [0,\infty),  D_{\rm f}(t) \geq D_{\rm c}(t)$ holds. Then $U_{\rm out}\left(t - D_{\rm c}(t)\right)$ reaches the moving shock front faster than $U_{\rm in}\left(t - D_{\rm f}(t)\right)$.
\end{remark}
We introduce a new coordinate $z$ defined as
\begin{align}
z =
\begin{dcases}
\frac{l(t)-x}{u}, & \quad  x \in[0,l(t)],  \\
\frac{x - l(t)}{u}, & \quad x \in[l(t),L], \\
\end{dcases}  \label{ztox}
\end{align}
and new variables $\tilde\varrho_{\rm f}(z,t)$ and $\tilde\varrho_{\rm c}(z,t)$ defined in $z$-coordinate. The transformations between $\tilde \rho_{\rm f}(x,t),\tilde \rho_{\rm c}(x,t)$ and $\tilde \varrho(z,t),\tilde \varrho_{c}(z,t)$ are given by
\begin{align}
\tilde \varrho_{\rm f}(z,t) =& \tilde \rho_{\rm f}(l(t)-uz,t), \quad z \in [0,D_{\rm f}(t)], \label{vrhof}\\
\tilde \varrho_{\rm c}(z,t) =& \tilde \rho_{\rm c}(l(t)+uz,t),  \quad \; z \in [0,D_{\rm c}(t)], \label{vrhoc}
\end{align}
and the associated inverse transformations of \eqref{vrhof} and \eqref{vrhoc} are given by
\begin{align}
\tilde \rho_{\rm f}(x,t) =& \tilde \varrho_{\rm f}\left(\frac{l(t)-x}{u},t\right), \quad x \in [0,l(t)],\\
\tilde \rho_{\rm c}(x,t) =& \tilde \varrho_{\rm c}\left(\frac{x - l(t)}{u},t\right), \quad x \in [l(t),L].
\end{align}
Using \eqref{vrhof} and \eqref{vrhoc}, the original system \eqref{lf}-\eqref{mi} is rewritten in the new $z$-coordinate as
\begin{align}
	\partial_t \tilde \varrho_{\rm f}(z,t) =& \left(1 - \frac{\dot l(t)}{u}\right) \partial_z \tilde \varrho_{\rm f}(z,t), \; z\in[0,D_{\rm f}(t)], \label{vlf}\\
	\partial_t \tilde \varrho_{\rm c}(z,t)  =& \left(1 + \frac{\dot l(t)}{u}\right)  \partial_z \tilde \varrho_{\rm c}(z,t), \; z\in[0,D_{\rm c}(t)], \label{vlc}\\
	\tilde \varrho_{\rm f}(D_{\rm f}(t),t) =& U_{\rm in}(t) , \label{vUin}\\
	\tilde \varrho_{\rm c}(D_{\rm c}(t),t) =& U_{\rm out}(t) , \label{vUout}
\end{align}
with the ODE given by  
\begin{align}\dot X(t)= - b \left( \tilde\varrho_{\rm f}(0,t) + \tilde\varrho_{\rm c}(0,t) \right).\label{vll}\end{align}

\subsection{Predictor-based backstepping transformation}

We consider the following backstepping transformation, motivated by the predictor-based transformation for delay representation $\varrho_{\rm f}(z,t)$ and $\varrho_{\rm c}(z,t)$ defined in \eqref{vlf}-\eqref{vUout},
\begin{align}
	w_{\rm f}(z,t)=&\tilde \varrho_{\rm f}(z,t) - K_{\rm f} \Bigg( X(t)  - b\int_{0}^{z} \tilde \varrho_{\rm f}(\xi,t)d\xi  \Bigg. \nonumber \\
	&\left.- b\int_{0}^{\min\{D_{\rm c}(t),z\}}\!\!\!\!\!\!\!\!\!\!\!\!\!\!\tilde \varrho_{\rm c}(\xi,t)d\xi \right), \quad z\in[0,D_{\rm f}(t)], \label{vrhotf}\\
	w_{\rm c}(z,t)=&\tilde \varrho_{\rm c}(z,t) - K_{\rm c} \Bigg( X(t) - b\int_{0}^{z} \tilde \varrho_{\rm c}(\xi,t)d\xi \Bigg.  \nonumber   \\
&\left. - b\int_{0}^{\min\{D_{\rm f}(t),z\}}\!\!\!\!\!\!\!\!\!\!\!\!\!\! \tilde \varrho_{\rm f}(\xi,t)d\xi \right), \quad z\in[0,D_{\rm c}(t)]. \label{vrhotc}
\end{align}
where $K_{\rm f}, K_{\rm c} >0$ are positive constant gain kernels. 

The above transformation in the original PDE state variables $\rho_{\rm f}(x,t)$ for $x \in [0,l(t)]$ and $\rho_{\rm c}(x,t)$ for $x \in [l(t),L]$, is given by
\begin{align}
	w_{\rm f}(x,t)=&\tilde \rho_{\rm f}(x,t) - K_{\rm f} \Bigg( X(t) - \frac{b}{u}\int_{x}^{l(t)} \tilde \rho_{\rm f}(\xi,t)d\xi  \Bigg. \nonumber \\
	&\left.  - \frac{b}{u}\int_{l(t)}^{\min\{L,2l(t)-x\}}\!\!\!\!\!\!\!\!\!\!\!\tilde \rho_{\rm c}(\xi,t)d\xi \right),\quad x \in [0,l(t)],\label{opt1} \\
	w_{\rm c}(x,t)=&\tilde \rho_{\rm c}(x,t) - K_{\rm c} \Bigg( X(t) - \frac{b}{u}\int_{l(t)}^{x}  \tilde \rho_{\rm c}(\xi,t)d\xi \Bigg.  \nonumber \\&\left. -\frac{b}{u}\int_{\max\{0, 2 l(t)-x\}} ^{l(t)}\!\!\!\!\!\!\!\!\!\! \tilde \rho_{\rm f}(\xi,t)d\xi \right), \quad x \in [l(t),L] .\label{opt2}
\end{align}
\begin{itemize}
\item For the case $D_{\rm f}(t) \leq D_{\rm c}(t)$, it follows that $l(t) \in \left[0,\frac{L}{2}\right]$ and  the following holds 
\begin{align}
 x \in [0,l(t)] \implies \min\{L,2l(t)-x\} = 2l(t)-x.
\end{align}
\item For the case $D_{\rm f}(t) \geq D_{\rm c}(t)$, it follows that $l(t) \in \left[\frac{L}{2},L\right]$, the following holds 
\begin{align}
x \in [l(t),L] \implies \max\{0, 2 l(t)-x\} = 2l(t)-x.
\end{align}
\end{itemize}
Later on, two pairs of state feedback controllers are obtained respectively for  $l(t) \in \left[0,\frac{L}{2}\right]$ and $l(t) \in \left[\frac{L}{2},L\right]$.
The inverse transformation of \eqref{opt1},\eqref{opt2}  is given by
\begin{align}
	\tilde \rho_{\rm f}(x,t) =& w_{\rm f}(x,t)  + K_{\rm f} \Bigg( X(t) - \frac{b}{u}\int_{x}^{l(t)} w_{\rm f}(\xi,t)d\xi  \Bigg. \nonumber \\
	&\left.  - \frac{b}{u}\int_{l(t)}^{\min\{L,2l(t)-x\}}\!\!\!\!\!\!w_{\rm c}(\xi,t)d\xi \right),\; x \in [0,l(t)],\label{inv1} \\
	\tilde \rho_{\rm c}(x,t) =& w_{\rm c}(x,t) + K_c \Bigg( X(t) - \frac{b}{u}\int_{l(t)}^{x}  w_{\rm c}(\xi,t)d\xi \Bigg.  \nonumber \\&\left. -\frac{b}{u}\int_{\max\{0, 2 l(t)-x\}} ^{l(t)}\!\!\!\!\!\!\!\!\!\!\! w_{\rm f}(\xi,t)d\xi \right), \quad x \in [l(t),L].\label{inv2}
\end{align}
Let us  denote the above transformations as
\begin{align}
	\tilde \rho_{\rm f} &= \mathcal{T}_{\rm f}[w_{\rm f},w_{\rm c}], \\
	\tilde \rho_{\rm c} &= \mathcal{T}_{\rm c}[w_{\rm f},w_{\rm c}].
\end{align}
At the moving interface, we have
\begin{align}
w_{\rm f}(l(t),t) =& \tilde \rho_{\rm f}(l(t),t) - K_{\rm f} X(t), \label{mv1}\\
w_{\rm c}(l(t),t) =& \tilde \rho_{\rm c}(l(t),t) - K_{\rm c} X(t). \label{mv2}
\end{align}
Taking temporal and spatial derivative on both sides of \eqref{opt1},\eqref{opt2} and substituting into the PDE-ODE original system \eqref{lf}-\eqref{mi}, we obtain target system by $w_{\rm f}(x,t)$ and $w_{\rm c}(x,t)$,
\begin{align}
\!\partial_t w_{\rm f} + u\partial_x w_{\rm f}  
=& \frac{K_{\rm f}  b}{u} \dot{l}(t) (g(t) + 2\epsilon_c(x,t)),\; x\in [0,l(t)], \label{t1}\\
\!\partial_t w_{\rm c} - u\partial_x w_{\rm c} 
=& \frac{K_{\rm c} b}{u}\dot{l}(t) (g(t) - 2\epsilon_{\rm f}(x,t)), \; x\in [l(t),L], \label{t2}\\
w_{\rm f}(0,t)=& 0 , \label{bcf} \\
w_{\rm c}(L,t)=& 0 , \label{bcc}\\
\dot X(t)= & -a X(t) - b \left(w_{\rm c}(l(t),t) + w_{\rm f}(l(t),t)\right), \label{def_ode}
\end{align}
where the constant coefficient $a = b(K_{\rm f} + K_{\rm c}) > 0$ is obtained by substituting \eqref{mv1},\eqref{mv2} into \eqref{mi},
given $b, K_{\rm f}, K_{\rm c} >0$. The time-varying term $g(t)$ is defined as
 \begin{align}
g(t)=& (K_{\rm f} - K_{\rm c}) X(t) + w_{\rm f}(l(t),t)-w_{\rm c}(l(t),t),  \label{g}
	\end{align}
and the  space and time-varying terms $\epsilon_{\rm c}(x,t)$ and $\epsilon_{\rm f}(x,t)$ are given by
\begin{align}
\notag	\epsilon_{\rm c}(x,t) =& \tilde \rho_{\rm c}(2l(t)-x,t)\\
	 =& \mathcal{T}_{\rm c}[w_{\rm f},w_{\rm c}](2l(t)-x,t), \label{epsic}\\
\notag	\epsilon_{\rm f}(x,t) =& \tilde \rho_{\rm f}(2l(t)-x,t) \\
=& \mathcal{T}_{\rm f}[w_{\rm f},w_{\rm c}](2l(t)-x,t). \label{epsif} 
\end{align}
We assume that densities outside freeway segment $[0,L]$ are at steady states, therefore $\tilde \rho_{\rm c} (2l(t)-x,t) = 0$ when $2l(t)-x >L$, and $\tilde \rho_{\rm f} (2l(t)-x,t) = 0$ when $ 2l(t)-x <0$. Hence,  the followings hold for $\epsilon_{\rm f}(x,t)$ and $\epsilon_{\rm c}(x,t)$,
\begin{align}
\begin{dcases}
\epsilon_{\rm f}(x,t) = 0, \quad  l(t) \in  [0, L/2]\;\; {\rm and}\;\;  x \in [2l(t),L],\\
\epsilon_{\rm c}(x,t) = 0,   \quad  l(t) \in  [L/2, L]\; {\rm and}\;\; x \in [0, 2l(t)-L].
\end{dcases} 
\end{align}
Otherwise, $\epsilon_{\rm f}(x,t)$ and $\epsilon_{\rm c}(x,t)$ are given by expressions in \eqref{epsic} and \eqref{epsif}.
The bilateral state feedback boundary actuations for inlet and outlet of the segment are derived  from \eqref{opt1},\eqref{opt2} and \eqref{bcf},\eqref{bcc} as 
\begin{align}
U_{\rm in}(t) = & K_{\rm f} \Bigg( X(t) - \frac{b}{u}\int_{0}^{l(t)} \tilde \rho_{\rm f}(\xi,t)d\xi \Bigg. \nonumber \\
&\left. - \frac{b}{u}\int_{l(t)}^{\min\{L,2l(t)\}}\tilde \rho_{\rm c}(\xi,t)d\xi \right) , \label{controlf} \\
U_{\rm out}(t) = & K_{\rm c} \Bigg( X(t) - \frac{b}{u}\int_{l(t)}^{L} \tilde \rho_{\rm c}(\xi,t)d\xi\Bigg.  \nonumber \\
&\left.-\frac{b}{u}\int_{\max\{0, 2 l(t)-L\}} ^{l(t)} \tilde \rho_{\rm f}(\xi,t)d\xi   \right) \label{controlc}.
\end{align}
We obtain two pairs of controller designs for $l(t) \in \left[0,\frac{L}{2}\right]$ and $l(t) \in \left[\frac{L}{2},L\right]$, respectively.
When $l(t) \in \left[0,\frac{L}{2}\right]$, it holds true that $\min\{L,2l(t)\} = 2l(t), {\max\{0, 2 l(t)-L\}} = 0$ and  when $l(t) \in \left[\frac{L}{2},L\right]$ one gets $\min\{L,2l(t)\} = L, \max\{0, 2 l(t)-x\} = 2l(t)$.

In addition, when $l(t) = \frac{L}{2}$, controller integral forms become identical  for $l(t) \in \left[0,\frac{L}{2}\right]$ and $l(t) \in \left[\frac{L}{2},L\right]$:
\begin{align}
	U_{\rm in}(t) = &  K_{\rm f}\! \left(\! X(t)- \frac{b}{u}\int_{0}^{\frac{L}{2}} \!\!\tilde \rho_{\rm f}(\xi,t)d\xi - \frac{b}{u}\int_{\frac{L}{2}}^{L}\!\!\tilde \rho_{\rm c}(\xi,t)d\xi\right)\!, \\
U_{\rm out}(t) = & {K_{\rm c}}\!\left(\! X(t) - \frac{b}{u}\int_{0}^{\frac{L}{2}}\!\! \tilde \rho_{\rm f}(\xi,t)d\xi - \frac{b}{u}\int_{\frac{L}{2}}^{L}\!\! \tilde \rho_{\rm c}(\xi,t)d\xi \right)\! .
\end{align}
It is remarkable that the bilateral control input smoothly switches between the above control laws when the moving interface position passes through the middle of the freeway segment.

Due to the invertibility of the transformation in \eqref{opt1},\eqref{opt2}, stability of the target system $(w_{\rm c}(x,t), w_{\rm f}(x,t), X(t))$ and stability  the plant  $(\tilde \rho_{\rm f}(x,t), \tilde \rho_{\rm c}(x,t), X(t))$ are equivalent. In the next section, we apply Lyapunov analysis to prove the stability of the target system. Define the $H^1$-norm $||f(\cdot, t)||_{H^1_{[a,b]}}$ as
\begin{align}
||f(\cdot,t)||_{H^1_{[a,b]}} = \left(\int_{a}^{b} f^2(x,t) + f_x^2(x,t) dx\right)^{1/2}.
\end{align}
We now state the main result of the paper.
\begin{theorem}
	Consider a closed-loop system consisting of the PDE-ODE system \eqref{rf}-\eqref{lt} and the bilateral full-state feedback control laws for inlet and outlet   \eqref{controlf},\eqref{controlc}.
	 For any system reference $(\rho_{\rm f}^\star, \rho_{\rm c}^\star, l^\star)$ which satisfies conditions \eqref{md1},\eqref{md2} and \eqref{md3} , and for any given $L>0$,
	there exist $c>0,\  \gamma >0,\  \zeta>0 $ such that if the initial conditions of the system $(\rho_{\rm f}(x,0),\rho_{\rm c}(x,0),l(0))$ satisfy $Z(0) < \zeta$, local exponential stability of the closed-loop system with bilateral control laws holds $\forall t\in[0,\infty)$, namely,
	\begin{align}
	 Z(t) \leq c e^{- \gamma t} Z(0),
	\end{align}
	 where $Z(t)$ is defined as
	 	\begin{align}
	 \notag		Z(t) =& || \rho_{\rm f}(x,t) -\rho_{\rm f}^\star||_{H^1_{[0,l(t)]}} + || \rho_{\rm c}(x,t) -\rho_{\rm c}^\star||_{H^1_{[l(t),L]}}\\
	 &+ |l(t) - l^\star|^2,
	 \end{align}
	 and condition \eqref{boundl} is satisfied for model validity. 
\end{theorem}
\section{Proof of Theorem 1}
In the proof, the local stability of the closed-loop system in the $H^1$ sense is shown with Lyapunov analysis and the following condition of model validity  \eqref{boundl} is guaranteed by our control design.
The proof of Theorem 1 is established through following steps: we firstly prove the local stability of the target system \eqref{t1}-\eqref{def_ode} for a given time interval $\forall t\in[0,t^\star)$ under the assumption that condition \eqref{boundl} is satisfied. Then we prove that with initial conditions of states variables bounded, the local exponential stability of the above target system holds for $\forall t\in[0,\infty)$ with the assumption removed. This is achieved by comparison principle and contradiction proof in Lemma 3. In the end, the stability analysis of the target system leads to stability of the original PDE-ODE system in \eqref{rf}-\eqref{lt}.

Let us define the  Lyapunov functional
\begin{align}
V(t) =&  V_1(t) +  V_2(t) +\lambda V_3(t) + \lambda V_4(t) +  V_5(t), \label{LV}
\end{align}
where $\lambda>0$ with the component Lyapunov functions  
\begin{align}
V_1(t) =&  \int_{0}^{l(t)} e^{{-x}} w_{\rm f}^2(x,t)dx,\\
V_2(t) =&  \int^{L}_{l(t)} e^{x-L} w_{\rm c}^2(x,t)dx,\\
V_3(t) =&  \int_{0}^{l(t)} e^{-x} \partial_x w_{\rm f}^2(x,t)dx,\\
V_4(t) =& \int^{L}_{l(t)} e^{x-L}\partial_x w_{\rm c}^2(x,t)dx,\\
V_5(t) =& X(t)^2.
\end{align}	
\begin{lemma}
	Assume $\exists t^\star > 0$ such that the condition in \eqref{boundl} is satisfied, then there exists $\sigma > 0$ such that the following holds $\forall t\in[0,t^\star)$,
	\begin{align}
		\dot V(t) \leq -  \sigma V + \tau V^{3/2}. \label{ls}
	\end{align}
\end{lemma}

\begin{proof}		
Taking time derivative of the Lyapunov function $\eqref{LV}$ along the solution of the target system \eqref{t1}-\eqref{def_ode}, we have
	\begin{align}
\notag\!\!\dot V_1(t)\! =& - u \int_{0}^{l(t)}\!\!\! \!e^{{-x}} w_{\rm f}^2(x,t)dx-\!(u\!-\!\dot l(t))e^{-l(t)} \! w_{\rm f} ^2 (l(t),t) \\ \notag & + \frac{2K_{\rm f}b}{u}\dot l(t) g(t) \int_{0}^{l(t)}e^{-x}w_{\rm f}(x,t)dx\\  & + \frac{4K_{\rm f}b}{u}\dot l(t)  \int_{0}^{l(t)} e^{-x}\epsilon_{\rm c}(x,t)w_{\rm f}(x,t)dx , \label{dV1}	\\
\notag	\dot V_2(t)\! =&- u \!\int^{L}_{l(t)}\!\!\!  e^{x-L} w_{\rm c}^2(x,t)dx - (u\! +\! \dot l(t))e^{l(t)-L} \!  w_c ^2 (l(t),t)\\ \notag &+ \frac{2K_{\rm c} b}{u}\dot l(t)g(t) \int^{L}_{l(t)}e^{x-L}w_{\rm c}(x,t)dx \\  &- \frac{4K_{\rm c} b}{u}\dot l(t)\int^{L}_{l(t)} e^{x-L}  \epsilon_{\rm f}(x,t) w_{\rm c}(x,t)dx ,
\end{align}
\begin{align}
\notag	\dot V_3(t) =& -u  \int_{0}^{l(t)} e^{-x}\partial_x  w_f^2(x,t)dx \\ 
\notag &-(u-\dot l(t))e^{-l(t)}  \partial_x w_{\rm f}^2 (l(t),t) + {u}\partial_x  w_{\rm f} ^2(0,t) \\  &+\frac{4K_{\rm f} b}{u} \dot l(t)\int_{0}^{l(t)}e^{-x} \partial_x \epsilon_{\rm c}(x,t) \partial_x w_{\rm f}(x,t)dx, \label{dV3}\\
\notag	\dot V_4(t) =& -u  \int^{L}_{l(t)} e^{{x-L}} \partial_x  w_{\rm c}^2(x,t)dx \\ 
\notag&- (u + \dot l(t))e^{l(t)-L}  \partial_x w_{\rm c} ^2 (l(t),t) + u\partial_x  w_{\rm c}^2(L,t)\\&-\frac{4K_{\rm c}b}{u}\dot l(t)\int^{L}_{l(t)}e^{x-L} \partial_x \epsilon_{\rm c}(x,t) \partial_x w_{\rm c}(x,t)dx ,\\
 \dot V_5(t) =& - a X(t)^2- b\left(w_{\rm c}(l(t),t) + w_{\rm f}(l(t),t)\right) X(t). \label{dV5}
\end{align}
By Agmon's inequality, the followings hold
	\begin{align}
	w_{\rm f} ^2 (l(t),t) &\leq ||w_{\rm f}||_\infty^2 \leq 4 ||\partial_x w_{\rm f}||_2^2 = 4 V_3, \label{wfl}\\
	w_{\rm c} ^2 (l(t),t) &\leq	||w_{\rm c}||_\infty^2 \leq 4 ||\partial_x w_{\rm c}||_2^2 = 4  V_4 \label{wcl}. 
	\end{align}
Plugging the above inequalities into the ODE \eqref{def_ode} yields that there exists $\delta>0$ such that
\begin{align}
|\dot l(t) | &<  a\sqrt{V_5} + b(\sqrt{V_3} + \sqrt{V_4}) <\delta \sqrt{V}. \label{bdotl}
\end{align}
	Using Young's inequality, Cauchy-Schwarz inequality for \eqref{g} and \eqref{wfl},\eqref{wcl}, we have 
	\begin{align}
	g(t) ^2
	&\leq \mu_1 V_3 + \mu_2 V_4 +  \mu_3  V_5, \label{gf}
	\end{align}
	where $\mu_j>0, j = 1, 2, 3$.
By definition of $\epsilon_{\rm c}(x,t)$ in \eqref{epsic}, there exist $\eta_k>0,\ k = 1, 2, 4$ such that 
\begin{align}
\int_{0}^{l(t)}\epsilon_{\rm c}^2(x,t)dx  &\leq \eta_1 V_1 + \eta_2 V_2 + \eta_4 V_4. \label{eta}
\end{align}
It follows that
\begin{align}
\notag	\dot V_1(t) \leq &  -{u}V_1 + |\dot l(t)| w_{\rm f}^2 (l(t),t) \\
 \notag &+ \frac{2K_{\rm f} b}{u}|\dot l(t)|\left(g^2(t) + \int_{0}^{l(t)} w_{\rm f}^2(x,t) dx \right)  \\ &+\frac{4K_{\rm f} b}{u} |\dot l(t)| \left(\int_{0}^{l(t)}\!\!\!\!\epsilon_{\rm c}^2(x,t)dx + \int_{0}^{l(t)}\!\!\!\!w_{\rm f}^2(x,t)dx\right),
\end{align}
Plugging \eqref{wfl} and \eqref{bdotl}-\eqref{eta} into the above inequality, there exists $\kappa_1>0$ such that 
\begin{align}
\dot V_1(t) \leq & - {u}  V_1  + \kappa_1 V^{3/2}, \label{dv1}
\end{align}
Taking total time derivative of boundary condition \eqref{bcf} yields,
\begin{align}
\partial_x w_{\rm f}(0,t) =& \frac{K_{\rm f} b}{u^2}\dot l(t) (g(t) + 2\epsilon_{\rm c}(0,t)),\label{x-w}
\end{align}	
Given definition of $\epsilon_{\rm c}(x,t)$ in \eqref{epsic}, there exist $\nu_0,\  \nu>0$ such that 
\begin{align}
	\epsilon_{\rm c}(0,t) &< \nu_0 V,\label{eps}\\
\int_0^{l(t)}\partial_x \epsilon^2_{\rm c} (x,t) &< \nu V.\label{epsx}
\end{align}
Using Young's inequality and plugging \eqref{gf} and \eqref{eps} into \eqref{x-w}, we obtain that there exists $\theta>0$ such that
{\begin{align}
\partial_x w_{\rm f}^2(0,t) &\leq \frac{K_{\rm c} b}{u^2} |\dot l(t)|\left(g^2(t) +  4\epsilon_{\rm c}^2(0,t)\right)
< \theta V^{3/2}, \label{wf0}
\end{align}}
Plugging \eqref{wfl}, \eqref{bdotl}, \eqref{epsx} and \eqref{wf0} into \eqref{dV3}, we obtain that there exists $\kappa_3>0$ such that
\begin{align}
	\dot V_3(t) \leq& - {u} V_3  + \kappa_3 V^{3/2},\label{dv3}
\end{align}
In the same fashion, we could obtain that there exist $\kappa_2, \kappa_4>0$ such that
\begin{align}
\dot V_2(t) \leq & - u V_2  + \kappa_2 V^{3/2}, \label{dv2}\\
\dot V_4(t) \leq& - u V_4 + \kappa_4 V^{3/2},\label{dv4}
\end{align}
For the last Lyapunov component, the following holds
\begin{align}	\dot V_5(t) \leq &-  {a} V_5 + \frac{4b}{a} V_3 +  \frac{4b}{a} V_4.\label{dv5}
	\end{align}
Using inequalities \eqref{dv1} and  \eqref{dv3}-\eqref{dv5} into  \eqref{LV},  it follows  that
	\begin{align}
	\notag \dot V(t)  \leq & -  u  V_1  -  u  V_2  -\left(\lambda u - \frac{4b}{a}\right)V_3\\
    & -\left(\lambda u- \frac{4b}{a}\right)V_4 - { a} V_5 + \tau V^{3/2}.
	\end{align}
	where $\tau = \kappa_1+ \kappa_2+\lambda\kappa_3+\lambda\kappa_4 >0$.
We choose $\lambda$ such that 
	\begin{align}
   \lambda>  \frac{4b}{au}, 
	\end{align}
	thus it holds that for $\sigma = \min \left\{ u -\frac{4 b}{\lambda a} ,a \right\}$,
	\begin{align}
\dot V(t) \leq -  \sigma V + \tau V^{3/2}.
\end{align}
\end{proof}
	\begin{lemma}
	According to  \eqref{ls}, for any $\sigma_0$ such that $0<\sigma_0<\sigma$, there exists $\delta_0>0$ such that for any $V(0) < \delta_0$,
	\begin{align}
	\tau |V^{3/2}| < (\sigma - \sigma_0)V 
	\end{align}	
	and,  
	\begin{align}
	\dot V(t) \leq -\sigma_0 V. \label{Vt}
	\end{align} 
	By comparison principle, the exponential stability is satisfied that $\forall t\in[0,t^\star)$,
	\begin{align}
	V(t) \leq V(0) e^{-\sigma_0 t}< \delta_0.
	\end{align}
\end{lemma}
\begin{lemma}
 	If the initial conditions of the target system $(w_{\rm f}(x,0),w_{\rm c}(x,0), X(0))$ satisfy the following 
 	\begin{align}
 		V(0)\leq \min \{\delta_0, \delta_1 \}, \label{V0}
 	\end{align}
 	where the positive constant $\delta_1$ is defined as
 	\begin{align}
 		\delta_1 =  \min \left\{(L-l^\star)^2, l^\star\right\}.\label{d1}
 	\end{align}
Then Lyapunov functional inequality \eqref{Vt} and condition \eqref{boundl} hold for $t \in [0,\infty)$.
 \end{lemma}

\begin{proof}
	We assume that there exists $t^\star > 0 $ such that  condition \eqref{boundl} is satisfied for $t \in [0,t^\star)$ but is violated at $t = t^\star$. Given \eqref{V0} and by comparison principle, the following inequality holds
	\begin{align}
		V(t^\star) \leq V(0) < \delta_1 \label{Vts}.
	\end{align}
	According to the definition of $V(t)$ in \eqref{LV}, we obtain that
	\begin{align}
	X^2(t^\star) &< V(t^\star).
	\end{align}
Combining \eqref{d1} and \eqref{Vts}, we have	
\begin{align}
	X^2(t^\star) &< \delta_1 =  \min \left\{(L-l^\star)^2, (l^\star)^2\right\}.\label{Xs}
\end{align}
Since $l(t^\star) = X(t^\star) + l^\star$ and $0<l^\star<L$, we obtain from \eqref{Xs} that
\begin{align}
	0 < l(t^\star) < L. \label{e2}
\end{align}
We conclude that \eqref{e2} contradicts the assumption that \eqref{boundl} is violated at $t = t^\star$. Therefore, the condition \eqref{boundl} is guaranteed for $t \in [0,\infty)$ when the initial condition $V(0)$ satisfies \eqref{V0}. This completes the proof Lemma 3.

Due to invertibility of the transformation in \eqref{opt1},\eqref{opt2},  we conclude that the system   \eqref{lf}-\eqref{mi} with control laws \eqref{controlf},\eqref{controlc} is locally exponentially stable in the $H^1$ norm, which completes the proof of Theorem 1. 
\end{proof}	

\section{Simulation}
\begin{figure}[t!]
	\centering
	\includegraphics[width=6.9cm]{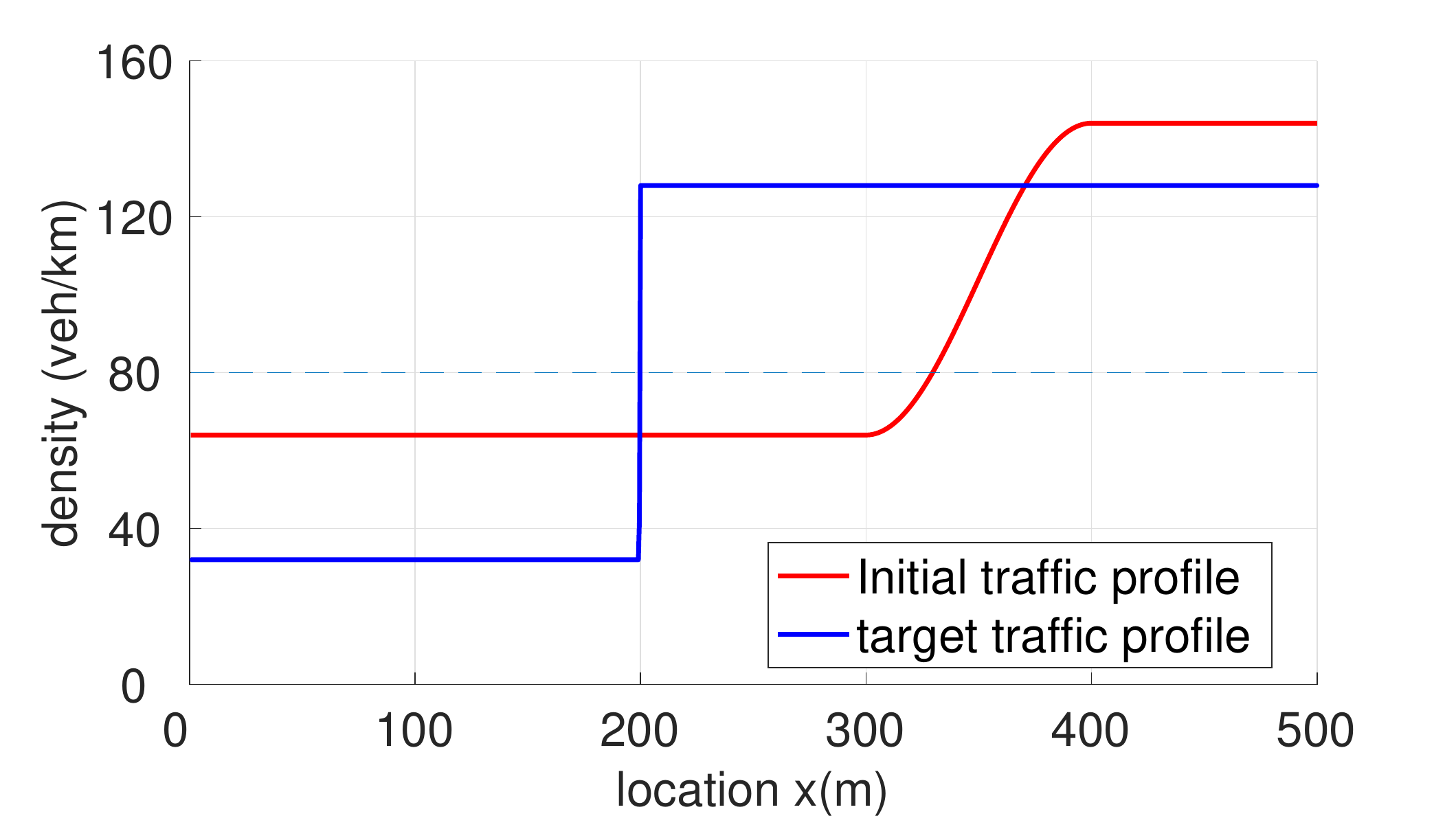}
	\caption{Traffic density profiles for initial condition with a soft shockwave and target system on freeway}\label{simu1}
\end{figure} 
\begin{figure}[t!]
	\centering
	\includegraphics[width=6.9cm]{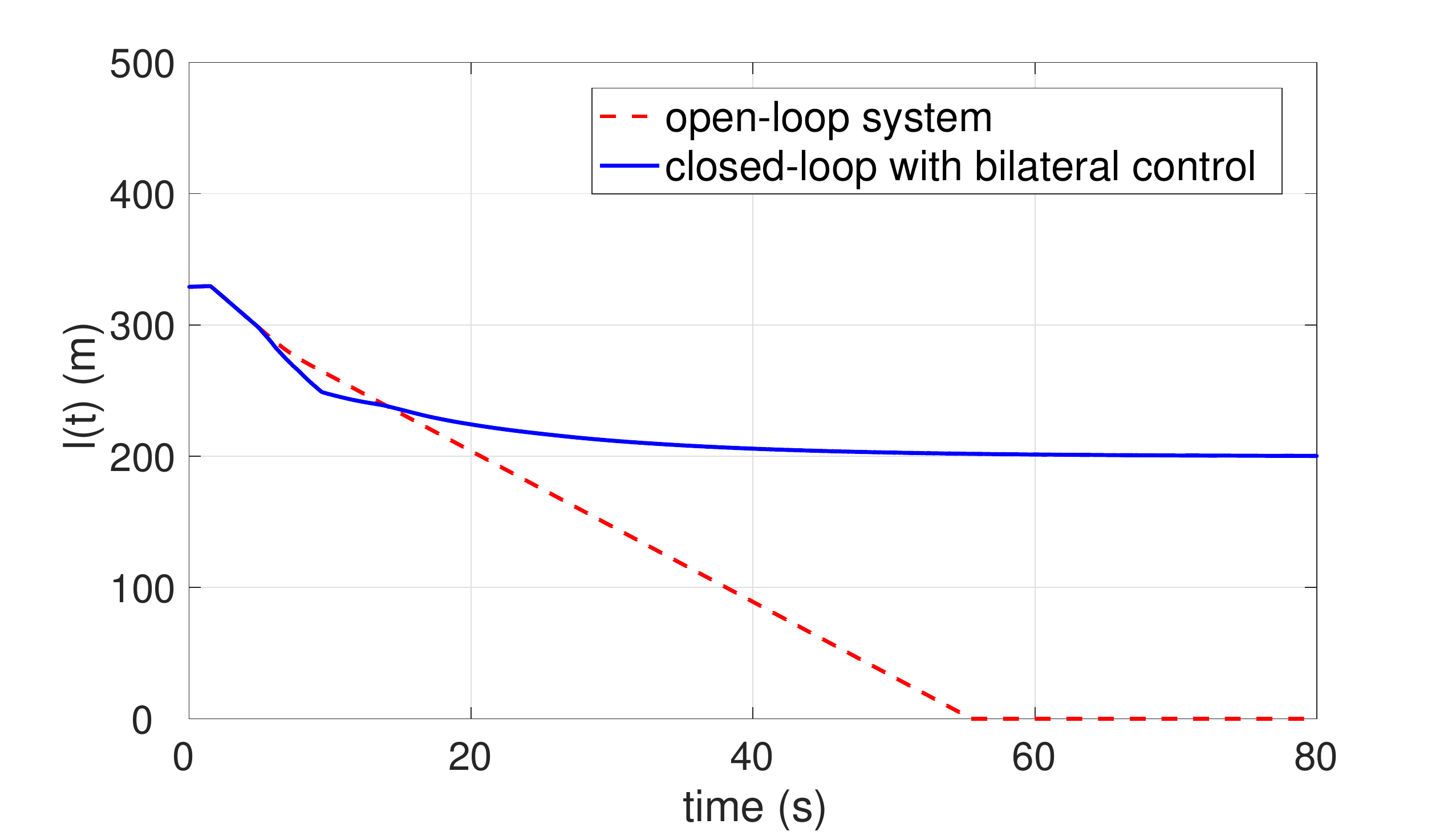}
	\caption{Evolution of the moving interface position $l(t)$ for open-loop system and for closed-loop system with bilateral boundary control}\label{simu2}
\end{figure} 
We simulate  proposed  control design considering a moving traffic shockwave in a $500$-meter freeway segment. The initial condition of the traffic profile and the desirable target traffic profile $\rho_{\rm f}^\star = 32 \; {\rm vehs/km}, \rho_{\rm c}^\star = 128 \; {\rm vehs/km} ,l^\star = 200\; {\rm m}, \rho_{\rm jump} = 80 \; {\rm vehs/km} $ are shown in Fig. \ref{simu1}, where the position of the shockwave front is initially located at $330$-meter and the final setpoint location is at $200$-meter. The initial position of the shockwave front is in the right-half plane of the segment while its final position is located at the left-half plane of the segment. The control objective is to regulate PDE states and ODE state from the initial profile to the reference profile. 

In Fig. \ref{simu2}, after around $40s$, the moving interface position stops at the setpoint location $l = 200$ m  with bilateral control  while in open-loop system it propagates upstream and travels out of the freeway segment before $1$ min. In Fig. \ref{simu3}, one can observe that  the bilateral control signals, the control inputs  also converge to zeros after around $40s$.
\begin{figure}[t!]
	\centering
\includegraphics[width=7cm]{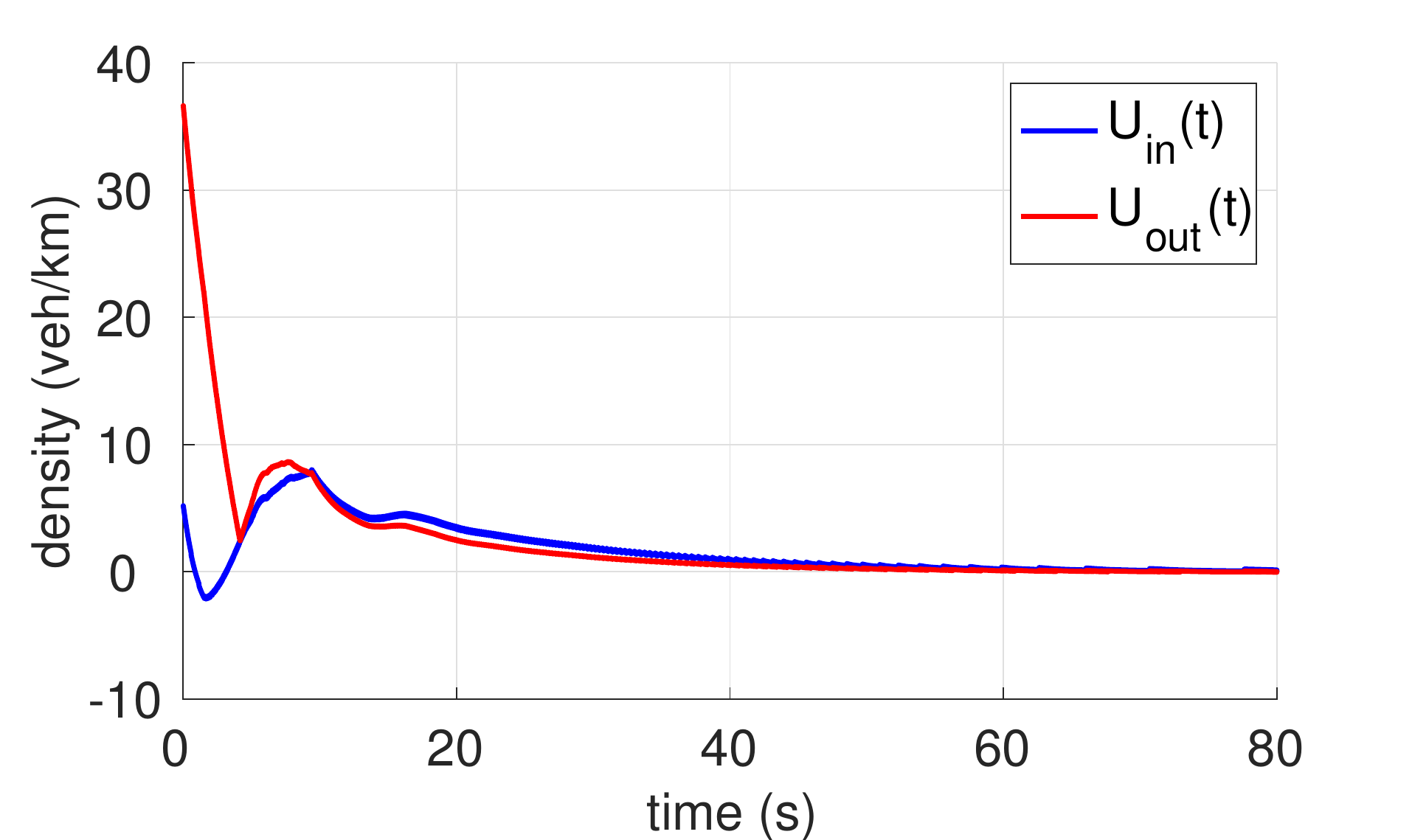}
	\caption{Evolution of bilateral control inputs over time}\label{simu3}
\end{figure} 

\section{Conclusion}
This paper addresses boundary feedback control problem of moving shockwave in congested traffic described by an PDE-ODE system. To stabilize the coupled system to a desired setpoint, we use predictor-based backstepping method to transform the state-dependent PDE-ODE coupled system to a target system, where the PDE state-dependent input delays to ODE are compensated by the bilateral boundary control inputs to PDEs.  Actuations of traffic densities at both boundaries are considered. The local exponential stability in $H^1$ norm is achieved and the model validity is guaranteed with the control designs. For future work, general theoretical results on multiple PDEs state-dependent input delays cascading to a nonlinear ODE is of authors' interest. 



\end{document}